\documentclass[12pt, reqno]{amsart}
\usepackage{amsmath, amsthm, amscd, amsfonts, amssymb, graphicx, color}
\usepackage{setspace}
\usepackage{multicol}
\usepackage[bookmarksnumbered, colorlinks, plainpages]{hyperref}
\hypersetup{colorlinks=true,linkcolor=red, anchorcolor=green, citecolor=cyan, urlcolor=red, filecolor=magenta, pdftoolbar=true}

\newtheorem{theorem}{Theorem}[section]
\newtheorem{lemma}[theorem]{Lemma}

\theoremstyle{definition}

\theoremstyle{remark}

\numberwithin{equation}{section}

\newcommand{\NN}{\mathbb{N}}

\newcommand{\CC}{\mathbb {C}}

\newcommand{\D}{\mathbb{D}}
\begin{document}
\setcounter{page}{1}
\title[Integral, Differential  and multiplication operators ] {Integral, differential  and multiplication operators on generalized Fock spaces }
\author [Tesfa  Mengestie]{Tesfa  Mengestie }
\address{Department of Mathematical Sciences \\
Western Norway University of Applied Sciences\\
Klingenbergvegen 8, N-5414 Stord, Norway}
\email{Tesfa.Mengestie@hvl.no}
\author[Sei-Ichiro Ueki]{Sei-Ichiro Ueki}
\address{Toki University \\
Hitachi316-8511, Japan}
\email{se-ueki@tokai.ac.jp}
\thanks{The first author   was  supported by HSH grant 1244/H15, and the  second author's work was partially supported by JSPS KAKENHI Grant 26800050. }

\begin{abstract}
 Volterra companion integral and multiplication operators with holomorphic symbols  are studied  for  a  large  class  of  generalized  Fock spaces on the complex plane $\CC$. The weights defining these spaces are radial and subject to a mild smoothness condition.  In addition,  we assumed that  the weights decay faster than the classical Gaussian weight. One of  our main results show  that there exists no nontrivial holomorphic   symbols $g$ which  induce  bounded  Volterra companion integral $I_g$ and multiplication operators $M_g$ acting between  the weighted  spaces. We also describe the bounded and compact Volterra-type integral operators $V_g$ acting between  $\mathcal{F}_q^\psi$ and
$\mathcal{F}_p^\psi$ when at least  one of the exponents $p$ or $q$ is infinite,  and extend results of Constantin and Pel\'{a}ez for  finite exponent cases. Furthermore, we showed that the differential operator $D$  acts in unbounded fashion   on  these and the   classical Fock spaces.
\end{abstract}

\maketitle

\section{Introduction} \label{1}
We denote by $\mathcal{H}(\CC)$ the space of  all entire functions on the complex plane $\CC$.
For   functions $f$ and $g$ in $\mathcal{H}(\CC),$  the  Volterra-type
integral operator $V_g$ and its companion operator $I_g$ with symbols $g$ are defined by
\begin{align*}
 V_gf(z)= \int_0^z f(w)g'(w) dw \ \ \ \text{and} \ \  \ I_gf(z)= \int_0^z f'(w)g(w) dw.
\end{align*}
Applying integration by parts in any one of the above integrals
gives the relation
 \begin{align}
\label{parts}V_g f+ I_g f= M_g f-f(0)g(0),
\end{align} where $M_g
f= g f$ is the multiplication operator of symbol $g$.   These
integral operators have been studied extensively on  various spaces of
analytic functions over several domains  with the aim to explore the connection between
their operator-theoretic behaviours with the function-theoretic properties of  the
symbols $g$,  especially after the works of Pommerenke \cite{Pom}, and
 Aleman and Siskakis \cite{Alsi1,Alsi2} on
Hardy and Bergman spaces.   For more information on the subject,
we refer to \cite{Alman,ALC,TM1,Si} and the related references
therein.

In \cite{JPP}, J. Pau and J. Pel\'{a}ez studied  some properties of the Carleson embedding maps and the  operator $V_g$ on
weighted Bergman spaces $A^p(w)$ over  the unit disc $\D$ when  $w$ belongs to a large class of rapidly decreasing weights. In \cite{Olivia},
Constantin and Pel\`{a}ez   modified the approaches in \cite{JPP} and \cite{Olivia1},  and studied the generalized Fock spaces $\mathcal{F}_p^\psi$ (see definition below) when the corresponding weight decays faster than the classical Gaussian weight. They obtained   several results including  complete characterizations of  the bounded, compact and trace ideal properties of the  operator $V_g$.   Interestingly, their results show that  there exists a much richer structure of $V_g$  on $\mathcal{F}_p^\psi$ than  when it acts on the classical Fock spaces $\mathcal{F}_p;$  the spaces which consist of all entire functions $f$ on $\CC$ for which
\begin{align*}
\int_{\CC}|f(z)|^p e^{-\frac{p}{2}|z|^2} dm(z) <\infty.
\end{align*}
In this paper, we  study some mapping properties of the operators $I_g$, $M_g$,  and the differential operator $D$ using the settings  in
\cite{Olivia}.  We will also consider the operator  $V_g$ for the cases where it has not been considered in \cite{Olivia}.  In contrast to   the case of  the operator $V_g$,  one of our results shows that there exists no richer structure of  $I_g$ and $M_g$  when they act on  the spaces $\mathcal{F}_p^\psi$ than  on the classical Fock  spaces $\mathcal{F}_p$. In some cases,  it rather shows poorer structure.
   From the  relation  in \eqref{parts},  we  also note in passing that  if any two of the operators $V_g, I_g$ and $M_g$ are bounded
so is the third one. In  generalized  Fock spaces,  more can be
said namely that $M_g$ is bounded (compact) if and only if  so is $I_g$.

 We shall thus first  set the setting as in \cite{Olivia}: we consider  a twice continuously differentiable function $\psi:[0, \infty) \to [0, \infty)$,   and for each point $z$ in $\CC$ we extend it to the whole  complex plane  by  setting $\psi(z)= \psi(|z|)$. We also assume that the Laplacian
 $\Delta \psi $  is positive and set
 $\tau(z)\simeq \footnote{The notation $U(z)\lesssim V(z)$ (or
equivalently $V(z)\gtrsim U(z)$) means that there is a constant
$C$ such that $U(z)\leq CV(z)$ holds for all $z$ in the set of a
question. We write $U(z)\simeq V(z)$ if both $U(z)\lesssim V(z)$
and $V(z)\lesssim U(z)$.} 1$ whenever $0\leq|z|<1$ and $\tau(z)\simeq (\Delta\psi(z))^{-1/2}$  otherwise,   where $\tau(z)$ is a radial differentiable function satisfying the conditions
  \begin{align}
  \label{assumption}
  \lim_{r \to \infty} \tau(r)= 0\ \ \text{and} \ \   \lim_{r\to \infty} \tau'(r)=0.
  \end{align}
   In addition, we require that either there exists a constant $C>0$ such that $\tau(r)r^C$ increases for large $r$ or
 \begin{align*}
 \lim_{r\to \infty} \tau'(r)\log\frac{1}{\tau(r)}=0.
 \end{align*}
 Throughout the paper we will  assume that $\psi$  and $\tau$ satisfy all the above  mentioned admissibility  conditions.
Observe that  there are many examples of weights $\psi$ that satisfy these  conditions. The  power functions as $\psi(r)= r^m, \ m>2$ and  the exponential type functions  $\psi(r)= e^{\alpha r},\ \ \alpha >0 $,  and  the  supper exponential functions $ \psi(r)= e^{e^{\alpha r}}, \  \alpha>0,$  are  all typical examples  of such weights.

The  generalized  Fock spaces $\mathcal{F}_p^\psi$ induced by the weight  function $\psi$ consist
of all entire functions $f$ for which
\begin{equation*}
\|f\|_{\mathcal{F}_p^\psi}^p=  \int_{\CC} |f(z)|^p
e^{-p\psi(z)} dm(z) <\infty,
\end{equation*} where  $0 < p <\infty,$  and  $dm$ denotes the
usual Lebesgue area  measure on $\CC$. For $p= \infty,$ the corresponding growth type  space $\mathcal{F}_{\infty}^\psi$  consists of all  entire functions $f$ such that
\begin{align*}
\|f\|_{\mathcal{F}_{\infty}^\psi}= \sup_{z\in \CC} |f(z)|e^{-\psi(z)} <\infty.
\end{align*}
\subsection{Integral type and multiplication operators}
We may mention that spaces of the form $\mathcal{F}_p^\psi$ were also studied earlier by other authors  with different contexts,  for example   in \cite{KH} when $p=2$ and $\psi$ belongs to a wider class of radial weights, and  in \cite{MOC} when $\psi$ is nonradial and its Laplacian  $\Delta\psi$ is  of a doubling measure. In \cite{Olivia}, conditions under which $V_g$ becomes bounded and compact when it acts between $\mathcal{F}_p^\psi$ and $\mathcal{F}_q^\psi$  for finite  exponents $p$ and $q$ were  obtained. Our first  main result extends  those  results when at least one of the exponent is  infinite.
\begin{theorem}\label{thm1}
Let $g$ be an entire function on $\CC$ and $0<p\leq \infty$.  Then
\begin{enumerate}
\item $V_g: \mathcal{F}_p^\psi \to \mathcal{F}_\infty^\psi $ is  bounded if and only if
\begin{align}
\label{sup}
\begin{cases}
& \sup_{z\in \CC} \frac{|g'(z)|}{1+\psi'(z)} <\infty,  \ \ \ \ \ \ \ p= \infty\\
& \sup_{z\in \CC} \frac{|g'(z)|(\Delta\psi(z))^{\frac{1}{p}}}{1+\psi'(z)} <\infty, \ \ \ 0<p<\infty.
\end{cases}
\end{align}
\item $V_g: \mathcal{F}_p^\psi \to \mathcal{F}_\infty^\psi $ is  compact  if and only if
\begin{align}
\label{V_gcompact}
\begin{cases}
&\lim_{|z|\to \infty } \frac{|g'(z)|}{1+\psi'(z)} =0, \ \ \ \ \ \ \ p= \infty\\
&\lim_{|z|\to \infty } \frac{|g'(z)|(\Delta\psi(z))^{\frac{1}{p}}}{1+\psi'(z)} =0, \ \ \ 0<p<\infty.
\end{cases}
\end{align}
\item if  $0<p<\infty$, then the following statements are equivalent.
\begin{enumerate}
\item $V_g: \mathcal{F}_\infty^\psi \to \mathcal{F}_p^\psi  $ is   bounded;
\item $V_g: \mathcal{F}_\infty^\psi \to \mathcal{F}_p^\psi  $ is  compact;
\item The function $\frac{g'}{1+\psi'}$ belongs to $L^p(\CC, dm)$.
\end{enumerate}
\end{enumerate}
\end{theorem}
For the special case when $p= \infty$, parts (i) and (ii) of the theorem are proved in  Theorem 3.4 and Theorem 3.5 of \cite{A}. We will provide    a different proof in section~\ref{3}.

In view of  Theorem~1.1, we  conclude  that there exists a richer structure of $V_g: \mathcal{F}_p^\psi \to \mathcal{F}_\infty^\psi $  than its action on  the classical setting where it was shown that $V_g$ is bounded  if and only if $g$  is a complex polynomial of degree not bigger than  two  \cite{TM0}. The same conclusion holds for boundedness or compactness of  $V_g: \mathcal{F}_\infty^\psi \to \mathcal{F}_p^\psi, \ p<\infty $ than its action on  the corresponding classical setting in which case boundedness of $V_g$ has been characterized by the fact that g is of  polynomial of degree not bigger than one \cite{B,TM0}. Some illustrative examples are the  following. The weight function  $\psi_\beta(z)= |z|^{\beta} \ \ \ \beta>2$ satisfies all the  initial admissibility  assumptions. Then  a consequence of the above result is  that the operator $V_g: \mathcal{F}_p^{\psi_\beta} \to \mathcal{F}_\infty^{\psi_\beta} $ is  bounded if and only if $g$ is a  complex polynomial of
\begin{align*}
deg (g) \leq
\begin{cases}
\beta, & p= \infty\\
 \ \\
\frac{\beta(p-1)+2}{p}, & \ \ 0<p<\infty.
\end{cases}
\end{align*}
On the other hand,  if  $\psi_\alpha(z)=e^{\alpha|z|}, \ \alpha >0$, which is also an admissible weight function, then  $V_g: \mathcal{F}_p^{\psi_\alpha} \to \mathcal{F}_\infty^{\psi_\alpha} $ is  bounded if and only if for all $z\in \CC$:
\begin{align*}
|g(z)| \lesssim \begin{cases}
e^{\alpha|z|}, & p= \infty\\
 \ \\
 e^{\frac{\alpha(p-1)}{p}|z|}, &\ \ \ 0<p<\infty.
 \end{cases}
 \end{align*}                                                                                                                                                                        Furthermore, if we, in particular,  take super exponential growth function   $\psi(z)= e^{e^{|z|}},$ then $\psi'(z) \simeq e^{e^{|z|}}$ and $\Delta\psi(z) \simeq e^{2|z|+e^{|z|}}$.  Simplifying condition \eqref{sup} shows that     $V_g: \mathcal{F}_p^{\psi} \to \mathcal{F}_\infty^{\psi} $ is  bounded if and only if  for all $z\in \CC:$
 \begin{align*}
|g'(z)| \lesssim \begin{cases}
e^{|z|}, & p= \infty\\
 \ \\
 e^{\frac{(p-2)}{p}|z|+ \frac{(p-1)}{p}e^{|z|} }, &\ \ \ 0<p<\infty.
 \end{cases}
 \end{align*}
An important ingredient used  in the proofs of  the results in \cite{Olivia}  when $V_g$ acts between the spaces $\mathcal{F}_p^\psi$ and $\mathcal{F}_q^\psi$   with finite exponents $p$ and $q$  has been  the descriptions of the Carleson and vanishing Carleson measures.  One could possibly follow similar approach to prove the above theorem as well.  It only requires to  describe the corresponding Carleson measures  first.  In  Section~\ref{3},  we will  give a direct proof of the theorem without being resorted to the notion of Carleson measures or embedding mapping techniques.

Our next main result describes the bounded and compact Volterra companion integral  operators $I_g$ and multiplication operators $M_g$ acting between the generalized  Fock spaces.
\begin{theorem}\label{thm2} Let $0<p,q \leq \infty $ and $g$ be an entire function on $\CC.$ Then
\begin{enumerate}
\item  if  $ p\neq q$, then the following statements  are equivalent.
\begin{enumerate}
\item $I_g: \mathcal{F}_p^\psi \to \mathcal{F}_q^\psi $ is  bounded;
\item $M_g: \mathcal{F}_p^\psi \to \mathcal{F}_q^\psi $  is bounded;
\item $g$ is the zero function.
\end{enumerate}
\item if   $0<p \leq \infty$, then the following  statements are equivalent.
\begin{enumerate}
\item $I_g: \mathcal{F}_p^\psi \to \mathcal{F}_p^\psi $ is  bounded;
\item $M_g: \mathcal{F}_p^\psi \to \mathcal{F}_p^\psi $ is  bounded;
\item $g$ is a constant  function.
\end{enumerate}
\item if   $0<p \leq \infty$, then the following  statements are also  equivalent.
\begin{enumerate}
\item $I_g: \mathcal{F}_p^\psi \to \mathcal{F}_p^\psi $ is compact;
\item $M_g: \mathcal{F}_p^\psi \to \mathcal{F}_p^\psi $  is  compact;
\item $g$ is the zero function.
\end{enumerate}
\end{enumerate}
 \end{theorem}
\vspace{-0.12in}
These results differ significantly   between the cases when $p= q $ and $p \neq q$. It has been known that this difference does not exist in the corresponding classical setting \cite{TM0}. On the other hand,  the appearance of such a difference has not been totally unexpected since in the  classical Fock spaces,  the monotonicity  property in the sense  of inclusion
$\mathcal{F}_p \subset \mathcal{F}_q$ whenever $0<p\leq q \leq \infty$,   holds \cite{SJR}. As follows from Corollary~2 of \cite{Olivia}, this  property  fails to hold for the family of generalized  Fock spaces $\mathcal{F}_p^\psi$. In fact, for finite $p$ and $q$ such that  $p\neq q$, it has been proved that
   \begin{align}
 \label{without}
 \mathcal{F}_p^\psi\setminus \mathcal{F}_q^\psi \neq \emptyset\  \ \text{and}\ \ \mathcal{F}_q^\psi\setminus \mathcal{F}_p^\psi \neq \emptyset.
 \end{align}
 As will be seen in the subsequent considerations, this property will be used in the proof of our results while in the classical setting the corresponding results were proved using the reproducing kernel $K_w$ as a  sequence of test functions  which rather belong to $\mathcal{F}_p$ for all possible exponents $0<p\leq \infty$.

 We remark that for $p= q= \infty$,  the Volterra-type integral  operator $V_g$  when $g$ is  the identity map and  the multiplication operator $M_g$ have been studied in \cite{A} in a more general setting  than ours, as the operators  in there act between two growth type spaces  where the weight functions defining the two spaces could be different.

\subsection{The differential operator $D$}
One striking feature of the  differential operator $Df= f'$ is that it is  a typical example of unbounded operators in many Banach spaces. In \cite{D,C}, conditions under which the operator becomes bounded on some growth type spaces of analytic functions have been given.  It turns out that the  operator  remains unbounded  when it acts on  generalized Fock spaces  with with weight decaying as at least as fast as the classical Gaussian weight.  We formulate this observation as follows.
\begin{theorem} \label{thm3}
Let $0<p,q \leq \infty$.  Then the differential operator
  $D:\mathcal{F}_p^\psi \to \mathcal{F}_q^\psi $ is
    unbounded. The same conclusion holds when $D$ acts  between  the classical Fock spaces.
\end{theorem}
For the special case when  $p= q= \infty$, the result follows from Theorem~2.10 of  \cite{D} or Theorem~4.1 of \cite{C}. Thus our contribution here is when at most one of the exponents is infinity. A different proof for $p= q= \infty$ will be also  provided  at the end of Section~\ref{diff}.
\section{Preliminaries}\label{2}
In this section,  we collect some known facts and auxiliary lemmas which  will be used in the sequel to  prove  our main results.
  Our  first lemma gives a complete characterization of the space $\mathcal{F}_\infty^\psi$ in terms of derivative.
\begin{lemma}\label{lem2} Let $f$ be a   holomorphic function on $\CC$. Then  $f$ belongs  to $\mathcal{F}_\infty^\psi$ if and only if
\begin{align}
\label{basic}
\sup_{z\in \CC} \frac{|f'(z)|e^{-\psi(z)}}{1+\psi'(z)}<\infty.
\end{align} In this case, we estimate  the norm  of $f$ by
\begin{align}
\label{infinite}
\|f\|_{\mathcal{F}_\infty^\psi} \simeq |f(0)|+ \sup_{z\in \CC} \frac{|f'(z)|e^{-\psi(z)}}{1+\psi'(z)}.
\end{align}
\end{lemma}
 The lemma follows from Corollary 3.3 of \cite{A}.  We give here a different proof which might be  interest of its own.
\begin{proof}
For a positive $r$ and entire function $f$, we denote its integral means by  \begin{align*}M_\infty(r, f)= \max_{|z|= r}|f(z)|.\end{align*}
 Then  $f$ belongs to
$\mathcal{F}_\infty^\psi$ if and only if \begin{align}
\label{max1}M_\infty(r,f)= O(e^{\psi(r)})\ \ \ \text{as}\ \ r\to \infty.
\end{align}
On the other hand, by  Lemma~21 of \cite{Olivia}, \ $M_\infty(r,f)= O(e^{\psi(r)})$ whenever  $r\to \infty$ if and only if
\begin{align}
\label{max2}
M_\infty(r,f')= O\bigg(\psi'(r)e^{\psi(r)}\bigg) \ \ \text {as}\ \ r\to \infty.
\end{align}
Furthermore, by our growth assumption, $\psi(r)$ grows faster than the classical Gaussian weight function $|r|^2/2$ and hence
\begin{align}
\label{large}1+\psi'(r)\simeq \psi'(r),
\end{align}
for sufficiently large $r$. From this,  our first  assertion on the lemma follows. \\
Next we prove the estimate in \eqref{infinite}.
We may observe that
\begin{align}
\label{oneside}
\|f\|_{\mathcal{F}_\infty^\psi} = \sup_{z\in \CC} |f(z)|e^{-\psi(z)} \geq \frac{1}{2}\bigg(\sup_{z\in \CC} |f(z)|e^{-\psi(z)} + |f(0)|e^{-\psi(0)}\bigg) \nonumber\\
\geq \frac{e^{-\psi(0)}}{2}\bigg(\sup_{z\in \CC} |f(z)|e^{-\psi(z)} + |f(0)|\bigg)\nonumber\\
\simeq \sup_{z\in \CC} |f(z)|e^{-\psi(z)} + |f(0)|.
\end{align} For $f$ in $\mathcal{F}_\infty^\psi$, condition \eqref{basic} along with \eqref{max1} and \eqref{max2} implies that the right-hand side of
\eqref{oneside} is bounded(up to a constant)  from below by
\begin{align*}
|f(0)|+ \sup_{z\in \CC} \frac{|f'(z)|e^{-\psi(z)}}{1+\psi'(z)}
\end{align*} from which one side of the estimate in \eqref{infinite} follows.  To prove the remaining estimate, we act as follows. Since $\|f\|_{\mathcal{F}_\infty^\psi} \leq |f(0)|+ \|f-f(0)\|_{\mathcal{F}_\infty^\psi}$, it suffices to show that $\|f-f(0)\|_{\mathcal{F}_\infty^\psi}$ is bounded by the quantity $\sup_{z\in \CC} \frac{|f'(z)|e^{-\psi(z)}}{1+\psi'(z)}.$ Thus, we write
\begin{align*}
|f(w)-f(0)|e^{-\psi(w)} \leq e^{-\psi(w)} \int_{0}^1 |w|f'(xw)dx\quad \quad \quad \quad \quad \quad \quad \quad \quad \quad \quad \quad \nonumber\\
\leq e^{-\psi(w)} \sup_{w\in \CC}\bigg( \frac{|f'(w)|e^{-\psi(w)}}{1+\psi'(w)}\bigg)
\int_{0}^1 |w|(1+\psi'(xw))e^{\psi(xw)}dx\nonumber\\
\lesssim e^{-\psi(w)} \bigg(\sup_{w\in \CC} \frac{|f'(w)|e^{-\psi(w)}}{1+\psi'(w)}\bigg) e^{\psi(w)}=\sup_{w\in \CC} \frac{|f'(w)|e^{-\psi(w)}}{1+\psi'(w)},
\end{align*} and completes the proof of the lemma.
\end{proof}
Note that the approximation formula \eqref{infinite} is in the spirit of the famous  Littlewood--Paley  formula for entire functions in the growth type  space
$\mathcal{F}_\infty^\psi$. The corresponding  formula for $\mathcal{F}_p^\psi$ for finite $p$ was obtained
in \cite{Olivia} and reads as  \begin{align}
\label{Paley}
\|f\|_{\mathcal{F}_p^\psi}^p \simeq |f(0)|^p + \int_{\CC} |f'(z)|^p \frac{e^{-p\psi(z)}}{(1+\psi'(z))^p} dm(z)
\end{align} for any entire function $f$. Both formulas \eqref{infinite} and \eqref{Paley} will be used repeatedly in our subsequent  considerations.
\begin{lemma}\label{lem3}
Let $0<q,p \leq\infty$ and $g$ be an  entire
function on $\CC$. Then
\begin{enumerate}
\item $V_{g}: \mathcal{F}_p^\psi\to
\mathcal{F}_q^\psi$ is compact if and only if \ $\|
V_{g}f_n\|_{\mathcal{F}_q^\psi} \to 0$ as $n\to \infty$ for each uniformly bounded
sequence $(f_n)_{n\in \NN}$ in  $\mathcal{F}_p^\psi$ converging
to zero uniformly on compact subsets of $\CC$ as $n\to \infty.$
\item  A similar statement holds when we replace the operator  $V_{g}$ by $I_{g}$ or $M_g$ in
(i).
\end{enumerate}
\end{lemma}
The lemma can be proved following standard arguments, and will  be used repeatedly in what follows  without mentioning
it over and over again.

We denote by $D(w,r)$ the Euclidean disk centered at $w$ and radius $r>0$.  Then we record the following useful covering lemma.
\begin{lemma}\label{lem4}
 Let $t: \CC \to (0,\infty)$ be a continuous function which satisfies  $|t(z)-t(w)| \leq \frac{1}{4}|z-w|$ for all $z$ and $w$ in $\CC$.  We also assume that $t(z) \to 0$ when $|z| \to \infty$. Then there exists a sequence of points $z_j$ in $\CC$ satisfying the following conditions.
 \vspace{-0.2in}
 \begin{enumerate}
 \begin{multicols}{2}
 \item $z_j\not\in D(z_k,t(z_k)), \ \ j \neq k$;
 \item $\CC= \bigcup_jD(z_j, t(z_j))$;
  \end{multicols}
 \vspace{-0.2in}
 \item $\bigcup_{z\in  D(z_j, t(z_j))}D(z, t(z)) \subset D(z_j, 3t(z_j))$;
 \item The sequence $ D(z_j, 3t(z_j))$ is a covering of $\CC$ with finite multiplicity $N_{\max}$.
 \end{enumerate}
\end{lemma}
This lemma was proved in \cite{Olivia} by adopting an approach used by  Oleinik \cite{Oleinik}. It will be used  in our subsequent proofs  being referred as  the covering lemma.

As pointed out earlier the reproducing kernel function $K_{w}(z)= e^{\langle z, w\rangle}$ has been used as a sequence of test functions to prove the corresponding results mentioned
above on the classical Fock spaces.   Unfortunately, an explicit expression for the reproducing kernel $K_{(w,\psi)}$ in  the generalized space $\mathcal{F}_2^\psi$ is still  unknown and it is not clear if any of the  arguments connected to  the reproducing kernels  in the classical  setting  could be directly carried over to the generalized  case. To prove our mains results, we will rather  use another  sequence of test functions in the current setting. This sequence has been used by several authors before for example \cite{Borch,Olivia,JPP}.  We introduce the test function as follows. By Proposition~A and Corollary~8 of \cite{Olivia},  for a sufficiently large positive number $R$, there exists a number $\eta(R)$ such that for any  $w\in \CC$ with $|w|> \eta(R)$, there exists an entire function $f_{(w, R)}$ such that
  \begin{enumerate}
  \item
    \begin{align}
  \vspace{-0.3in}
  \label{test00}
      |f_{(w,R)}(z)| e^{-\psi(z)}\leq C \min\Bigg\{ 1,\bigg(\frac{\min\{\tau(w), \tau(z)\}}{|z-w|}\bigg)^{\frac{R^2}{2}}\Bigg\}  \ \ \ \ \ \ \end{align}for all  $ z\in \CC$ and  for some constant $C$ that depends on $\psi$ and $R$. In particular when $z \in D(w, R\tau(w))$, the estimate becomes
   \begin{align}
  \label{test0}
   |f_{(w,R)}(z)| e^{-\psi(z)}\simeq 1.
    \end{align}
  \item $f_{(w, R)}$ belongs to $\mathcal{F}_p^\psi$  and its norm is estimated by
\begin{align}
\label{test}
\| f_{(w,R)}\|_{\mathcal{F}_p^\psi}^p \simeq \tau(w)^2,\ \ \ \  \eta(R) \leq |w|
\end{align} for all  $p$ in the range $0<p<\infty$.

Another important ingredient in our subsequent consideration is the pointwise estimate for subharmonic functions $f$, namely that
\begin{align}
\label{pointwise}
|f(z)|^p e^{-\beta \psi(z)} \lesssim \frac{1}{\sigma^2\tau(z)^2} \int_{D(z, \sigma \tau(z))} |f(w)|^pe^{-\beta\psi(w)} dm(w)
\end{align} for all finite exponent $p$,  any real number $\beta$,  and  a small positive number $\sigma$: see Lemma~7 of \cite{Olivia} for more details.
 \end{enumerate}
\begin{lemma}\label{lem5}
Let $R$ be a sufficiently large number and  $\eta(R)$ be as before. If $(z_k)$  is the  covering sequence from Lemma~\ref{lem4}, then  the function
\begin{align*}
F= \sum_{z_k: |z_k| > \eta(R)} a_k f_{(z_k, R)}
\end{align*} belongs to $\mathcal{F}_\infty^\psi$ for every $\ell^\infty$  sequence $(a_k)$,  and also $\|F\|_{\mathcal{F}_\infty^\psi} \lesssim \|(a_k)\|_{\ell^\infty}$.
\end{lemma}
\begin{proof}
We estimate the norm of $F$ as
\begin{align}
\label{newlem}
\|F\|_{\mathcal{F}_\infty^\psi}=
\sup_{z\in \CC} |F(z)|e^{-\psi(z)} \lesssim \sup_{z\in \CC} \sum_{k} |a_k ||f_{(z_k, R)}(z)|e^{-\psi(z)}\ \ \ \quad \quad \quad \ \ \ \nonumber \\
\leq \|(a_k)\|_{\ell^\infty} \sup_{z\in \CC} \sum_{k} |f_{(z_k, R)}(z)|e^{-\psi(z)}.
\end{align}
Invoking \eqref{test00}, the right-hand side of \eqref{newlem} is bounded by
\begin{align*}
\|(a_k)\|_{\ell^\infty} \sup_{z\in \CC}\Bigg( \sum_{k: z_k\neq z} \frac{\tau(z_k)^{\frac{R^2}{2}}}{|z_k-z|^{\frac{R^2}{2}}}+ \sum_{k: z_k= z} 1 \Bigg) \quad \quad \quad \ \quad \quad \quad \ \quad \quad \quad  \ \nonumber\\
\leq \|(a_k)\|_{\ell^\infty} \bigg(\sup_k \tau(z_k)^{\frac{R^2}{2}} \sup_{z\in \CC} \sum_{k: z_k\neq z} \frac{1}{|z_k-z|^{\frac{R^2}{2}}} +N_{\max}\bigg)\nonumber\\
\lesssim \|(a_k)\|_{\ell^\infty} \bigg(\sup_k \tau(z_k)^{\frac{R^2}{2}} \sup_{z\in \CC}  \frac{1}{|z_{k_0}-z|^{\frac{R^2}{2}}}+N_{\max} \bigg),
\end{align*} where $k_0$ is  the index for which $|z_{k_0}|\leq |z_{k}|$ for all $z_k \neq z$ and $N_{\max}$ as in Lemma~\ref{lem4}.   Observe that  because of \eqref{assumption}, $\sup_k \tau(z_k)$ is finite and hence
\begin{align*}
\|(a_k)\|_{\ell^\infty} \bigg(\sup_k \tau(z_k)^{\frac{R^2}{2}} \sup_{z\in \CC}  \frac{1}{|z_{k_0}-z|^{\frac{R^2}{2}}}+N_{\max} \bigg) \lesssim \|(a_k)\|_{\ell^\infty},
\end{align*} and completes the proof.
\end{proof}
We note  that the finite exponent version of the above lemma was proved in \cite[Proposition~9]{Olivia}. That is,  the function
\begin{align}
\label{finitediscrete}
F= \sum_{z_k:|z_k|\geq\eta(R)} a_k \frac{f_{(z_k,R)}}{\tau(z_k)^{\frac{2}{p}}}
\end{align} belongs to $\mathcal{F}_\psi^p$ for every $\ell^p$ summable sequence $(a_k)$  with norm  estimated by
\begin{align}
\label{finitediscrete0}
\|F\|_{\mathcal{F}_\psi^p}^p \lesssim \sum_{k} |a_k|^p.
\end{align}
\begin{lemma} \label{lemanew}
Let $h$ be a holomorphic function on $\CC$ and $p$ and $q$ be to positive numbers.  Then $|h(z)|\lesssim \tau(z)^{2\frac{q-p}{p}}$ for all $z\in \CC$ if and only if \begin{align}
\label{newfor}
\sup_{w\in \CC} \frac{1}{\tau(w)^{\frac{2q}{p}}} \int_{D(w, \sigma\tau(w))} |h(z)|^q dm(z) <\infty.
\end{align}
\end{lemma}
\begin{proof}
By Lemma~5 of \cite{Olivia}, for each  $w\in \CC$ and $z\in D(w, \sigma\tau(w))$ it holds that
\begin{align}\label{compare}
\tau(w) \simeq \tau(z)
\end{align} for a small positive number   $\sigma $.
The proof is then  an immediate consequence of this estimate.  We include a little proof for completeness. If $ |h(z)|\lesssim  \tau(z)^{2\frac{q-p}{p}}$, then by \eqref{compare}
\begin{align*}
\sup_{z\in \CC} \frac{1}{\tau(z)^{2q/p}}\int_{D(z, \sigma \tau(z))} |h(w)|^q dm(w) \simeq \sup_{z\in \CC} \int_{D(z, \sigma\tau(z))}\frac{ |h(w)|^q}{\tau(w)^{2q/p}} dm(w)\nonumber\\
\lesssim \sup_{w\in \CC}\frac{|h(w)|^q}{\tau(w)^{2q/p}}  \int_{D(z, \sigma \tau(z))} dm(w)\lesssim \sup_{w\in \CC} \frac{|h(w)|^q} {\tau(w)^{2\frac{q-p}{p}}}<\infty.
\end{align*}
On the other hand if \eqref{newfor} holds, then by subharmonicity of $|h|^q$ and \eqref{pointwise} we have
\begin{align*}
|h(w)|^q  \lesssim \frac{1}{\sigma^{2q}\ \tau(w)^{2q}} \int_{D(w, \sigma \tau(w))}|h(z)|^q dm(z)
\end{align*} from which it follows that
 \begin{align*}
|h(w)|^q \tau(w)^{2-\frac{2q}{p}} \lesssim \tau(w)^{-\frac{2q}{p}} \int_{D(w, \sigma \tau(w))}|h(z)|^q dm(z)
\end{align*} and the assertion follows.
\end{proof}
\section{Proof of the main results}\label{3}
We now turn to the proofs of the main results of the paper.
\subsection{Proof of Theorem~\ref{thm1}}
We begin with the proof of  the sufficiency of the condition in part (i). If $p= \infty$, then applying \eqref{infinite}
\begin{align*}
\|V_g f\|_{\mathcal{F}_\infty^\psi }\simeq\sup_{z\in \CC} \frac{|f(z)g'(z)|}{1+\psi'(z)}e^{-\psi(z)}
\leq \bigg(\sup_{z\in \CC} \frac{|g'(z)|}{1+\psi'(z)} \bigg) \sup_{z\in \CC} |f(z)|e^{-\psi(z)}\nonumber\\
= \|f\|_{\mathcal{F}_\infty^\psi }\sup_{z\in \CC} \frac{|g'(z)|}{1+\psi'(z)},
\end{align*}  from which it  also  follows that
\begin{align}
\label{norm1}
\|V_g\| \lesssim \sup_{z\in \CC} \frac{|g'(z)|}{1+\psi'(z)}.
\end{align}
On the other hand, if $0<p<\infty$, then applying \eqref{pointwise}, we have
\begin{align*}
\|V_g f\|_{\mathcal{F}_\infty^\psi }\simeq\sup_{z\in \CC} \frac{|g'(z)|}{1+\psi'(z)}|f(z)|e^{-\psi(z)} \ \quad \quad \quad \quad \quad \quad \quad \quad  \quad \quad \quad \quad\nonumber \\
\lesssim\sup_{z\in \CC} \frac{|g'(z)|}{1+\psi'(z)}\bigg(\frac{1}{\sigma^2\tau(z)^2}\int_{D(z,\sigma \tau(z))} |f(w)|^p e^{-p\psi(w)}dm(w)\bigg)^{\frac{1}{p}}\nonumber\\
\lesssim\sup_{z\in \CC} \frac{|g'(z)|\|f\|_{\mathcal{F}_p}}{(1+\psi'(z))\tau(z)^{\frac{2}{p}}}\simeq \|f\|_{\mathcal{F}_p} \sup_{z\in \CC} \frac{|g'(z)| (\Delta\psi(z))^{\frac{1}{p}}}{1+\psi'(z)}
\end{align*} as required and also deduce  the reverse   estimate  in \eqref{norm1}
\begin{align}
\label{norm1n}
\|V_g\| \lesssim \sup_{z\in \CC} \frac{|g'(z)|(\Delta\psi(z))^{\frac{1}{p}}}{1+\psi'(z)}.
\end{align}
To prove the necessity, note that from \eqref{test}, the sequence of  functions  $f_{(w, R)}$ belong to $\mathcal{F}_p^\psi $  for all $p <\infty$  and $w\in \CC$. On the other hand for $p= \infty$, applying \eqref{test00} and  \eqref{test0},  we make the corresponding estimate
\begin{align}
\label{norm2}
\|f_{(w, R)}\|_{\mathcal{F}_\infty^\psi}\simeq 1.
\end{align}
Applying $V_g$ to $f_{(w, R)}$ and making use of  Lemma~\ref{lem2} and \eqref{norm2},  we find
\begin{align*}
 \|V_g\| \gtrsim \|V_g f_{(w,R)}\|_{\mathcal{F}_\infty^\psi}\simeq\sup_{z\in \CC} \frac{|f_{(w,R)}(z)g'(z)|}{1+\psi'(z)}e^{-\psi(z)} \geq \frac{|f_{(w,R)}(z)g'(z)|}{1+\psi'(z)}e^{-\psi(z)}
\end{align*} for all $w, \ z\in \CC.$ In particular,  setting $z=w$ and making use of \eqref{test0} give
\begin{align}
\label{infty2}
 \frac{|g'(w)|}{1+\psi'(w)} \simeq \frac{|g'(w)|}{1+\psi'(w)}|f_{(w,R)}(w)|e^{-\psi(w)}\lesssim \|V_g\|,
\end{align}  and   the necessity and the reverse estimate in  \eqref{norm1} follow for the case $p= \infty$. Seemingly, for the remaining case where $0<p<\infty$, the estimate in \eqref{test0} and \eqref{test}   imply
\begin{align*}
 \|V_g\| \gtrsim\frac{1}{\tau(w)^{\frac{2}{p}}} \|V_g f_{(w,R)}\|_{\mathcal{F}_\infty^\psi}\simeq \frac{1}{\tau(w)^{\frac{2}{p}}}\sup_{z\in \CC} \frac{|f_{(w,R)}(z)g'(z)|}{1+\psi'(z)}e^{-\psi(z)}\ \quad \quad \quad \quad \\
 \geq \frac{|f_{(w,R)}(w)g'(w)|}{\tau(w)^{\frac{2}{p}}(1+\psi'(w))}e^{-\psi(w)} \simeq \frac{|g'(w)|}{\tau(w)^{\frac{2}{p}}(1+\psi'(w))}\simeq
 \frac{|g'(w)| (\Delta \psi(w))^{\frac{1}{p}}}{(1+\psi'(w))} .
\end{align*}
Observe that  this together with \eqref{norm1}, \eqref{norm1n} and \eqref{infty2} give, in addition,  an  approximation formulas for the operator norm:
\begin{align*}
\|V_g\| \simeq \begin{cases} & \sup_{z\in \CC} \frac{|g'(z)|}{1+\psi'(z)}, \ \  \ \  \ \ p= \infty\\
& \sup_{z\in \CC} \frac{|g'(z)|(\Delta \psi(z))^{\frac{1}{p}}}{1+\psi'(z)}, \ \  \ \  \ \ 0<p< \infty.
\end{cases}
\end{align*}
\emph{Part (ii)}. We  first prove the sufficiency of the condition. Let $f_n$ be a uniformly bounded sequence of functions in $\mathcal{F}_p^\psi$ that converges uniformly to zero on   compact subsets of  $\CC$. We let first $p= \infty$. Then for each positive $\epsilon$, the necessity of the condition  implies that there exists $N_1$ such that
\begin{align*}
\frac{|g'(z)| }{1+\psi'(z)} <\epsilon
\end{align*} whenever $|z|>N_1.$  From this and Lemma~\ref{lem2}, we have
\begin{align*}
\frac{|g'(z) f_n(z)|e^{-\psi(z)}}{1+\psi'(z)} \lesssim \| f_n\|_{\mathcal{F}_\infty^\psi}\frac{|g'(z)|}{1+\psi'(z)}\lesssim \epsilon
\end{align*} for all $|z|>N_1$.  We need to conclude the same for all $z$ such that $|z|\leq N_1$. To this end, we may first observe that the function $f^*(z)= 1$ belongs to $\mathcal{F}_p^\psi$ for all $0<p\leq\infty$. Since \eqref{V_gcompact} obviously implies the boundedness condition in part (i), it follows that
\begin{align*}
\|V_g f^*\|_{\mathcal{F}_\infty^\psi}\simeq \sup_{z\in \CC} \frac{|g'(z)| e^{-\psi(z)}}{1+\psi'(z)} <\infty.
\end{align*}
A consequence of this is that
\begin{align}
\label{forall}
\sup_{|z|\leq N_1}\frac{|g'(z) f_n(z)|e^{-\psi(z)}}{1+\psi'(z)} \lesssim \sup_{|z|\leq N_1}| f_n(z)|\to 0, \ \ \ n \to \infty.
\end{align}
 Similarity, when $0<p<\infty$,  the condition implies that for some positive number $N_2$  we have
 \begin{align*}
 \sup_{|z| > N_2}\frac{|g'(z)| }{\tau(z)^{\frac{2}{p}}} <\epsilon
 \end{align*} from which and \eqref{pointwise} we also have
  \begin{align*}
 \sup_{|z| > N_2} \frac{|g'(z)||f_n(z)|e^{-\psi(z)}}{1+\psi'(z)} \lesssim \sup_{|z| > N_2} \frac{|g'(z)|}{1+\psi'(z)} \bigg( \frac{1}{\tau(z)^2} \int_{D(z, \sigma\tau(z))}\frac{|f_n(w)|}{e^{-p\psi(w)}} dm(w)\bigg)^{\frac{1}{p}}\nonumber\\
 \lesssim  \sup_{|z| > N_2} \frac{|g'(z)|}{1+\psi'(z)}  \frac{1}{\tau(z)^{\frac{2}{p}}}\|f_n\|_{\mathcal{F}_p}
 \lesssim  \sup_{|z| > N_2} \frac{|g'(z)|}{1+\psi'(z)}  \frac{1}{\tau(z)^{\frac{2}{p}}} \leq \epsilon.
  \end{align*}
On the other hand, when $|z| \leq N_2$, the desired conclusion follows from  \eqref{forall}.

 Conversely, let us now show that if  $V_g: \mathcal{F}_p^\psi \to \mathcal{F}_\infty^\psi $ is compact, then the relation in \eqref{V_gcompact} holds. To this end, we take the test function
\begin{align}
\label{again}
f_{(w,R)}^*=\begin{cases} \frac{f_{(w,R)}}{\tau(w)^{2/p}}, \quad \ \eta(R) \leq |w|, \ \ \text{for}\ \   p <\infty\\
            f_{(w,R)} \ \ \ \ \text{for}\ \   p =\infty,
            \end{cases}
\end{align}
where $f_{(w,R)}$  is the function described with properties  in \eqref{test00}, \eqref{test0} and \eqref{test}. As shown in \cite{Olivia},
$f_{(w,R)}^* \to 0 $ as $ |w| \to \infty$ uniformly on compact subsets of $\CC$ and
\begin{align}
\label{need}
\sup_{|w|\geq \eta(R)} \|f_{(w,R)}^*\|_{\mathcal{F}_p^\psi}<\infty.
\end{align} It follows from this and compactness of $V_g$ that
\begin{align*}
\lim_{|w| \to \infty}\| V_g f_{(w,R)}^*\|_{\mathcal{F}_\infty^\psi}= 0.
\end{align*}
Making use of this and  \eqref{test0}, for $p= \infty$ we obtain
\begin{align*}
\lim_{|z| \to \infty}\frac{|g'(z)|}{1+\psi'(z)} \simeq \lim_{|z| \to \infty} \frac{|g'(z)|}{1+\psi'(z)}|f_{(z,R)}^*(z)| e^{-\psi(z)} \lesssim \lim_{|z| \to \infty}\| V_g f_{(z,R)}^*\|_{\mathcal{F}_\infty^\psi}=0.
\end{align*}
On the other hand, if $0<p<\infty$, then
\begin{align*}
\lim_{|z| \to \infty}\frac{|g'(z)|\Delta\psi(z)^{\frac{1}{p}}}{1+\psi'(z)}\simeq \lim_{|z| \to \infty} \frac{|g'(z)|}{1+\psi'(z)}|f_{(z,R)}^*(z)| e^{-\psi(z)} \lesssim \lim_{|z| \to \infty}\| V_g f_{(z,R)}^*\|_{\mathcal{F}_\infty^\psi}=0
\end{align*}
 and completes the proof of part (ii) of  Theorem~\ref{thm1}.

 \emph{Part (iii)}. Since (b) obviously  implies (a), it  suffices to show   (a) implies (c) and (c) implies (b). For the first, we follow this classical technique  where the original idea goes back to Luecking \cite{DL}.  Let $0<q<\infty$ and $R$  be a sufficiently large number   and $(z_k)$ be the covering sequence as in Lemma~\ref{lem4}. Then by  Lemma~\ref{lem5},
\begin{align*}
F= \sum_{z_k:|z_k|\geq\eta(R)} a_k f_{(z_k,R)}
\end{align*} belongs to $\mathcal{F}_\infty^\psi$ for every $\ell^\infty$  sequence $(a_k)$  with norm  estimate
$
\|F\|_{\mathcal{F}_\infty^\psi} \lesssim \|(a_k)\|_{\ell^\infty}.
$
If $(r_k(t))_k$  is the Radmecher sequence of function on $[0,1]$ chosen as in \cite{DL}, then the sequence $(a_kr_k(t))$ also
belongs to $\ell^\infty$ with $\|(a_kr_k(t))\|_{\ell^\infty}= \|(a_k)\|_{\ell^\infty}$ for all $t$. This implies that the function
\begin{align*}
F_t= \sum_{z_k:|z_k|\geq\eta(R)} a_k r_k(t) f_{(z_k,R)}
\end{align*} belongs to $\mathcal{F}_\infty^\psi$  with norm estimate
$\|F_t\|_{\mathcal{F}_\infty^
\psi} \lesssim \|(a_k)\|_{\ell^\infty}. $ Then, an  application of Khinchine's inequality \cite{DL} yields
\begin{align}
\label{Khinchine1}
\Bigg(\sum_{z_k:|z_k|\geq\eta(R)} |a_k|^{2}|f_{(z_k,R)}(z)|^2\Bigg)^{\frac{q}{2}}\lesssim \int_{0}^1\bigg| \sum_{z_k:|z_k|\geq\eta(R)} a_k r_k(t) f_{(z_k,R)}(z)\bigg|^q dt.
\end{align}
Setting  $
d\theta_{(g, \psi, q)}(z)= |g'(z)|^q  e^{-q\psi(z)}(1+\psi'(z))^{-q}dm(z),
$
 making use of \eqref{Khinchine1}, and subsequently Fubini's theorem, we have
 \begin{align}
\int_{\CC}\Bigg(\sum_{z_k:|z_k|\geq\eta(R)} |a_k|^{2} |f_{(z_k,R)}(z)|^2\Bigg)^{\frac{q}{2}}& d\theta_{(g, \psi)}(z)\quad \quad \quad \quad \quad  \quad \quad \quad \quad \quad  \quad \quad \quad \quad \quad  \nonumber\\
&\lesssim \int_{\CC} \int_{0}^1\bigg| \sum_{z_k:|z_k|\geq\eta(R)} a_k r_k(t) f_{(z_k,R)}(z)\bigg|^q dt d\theta_{(g, \psi, q)}(z)\nonumber\\
&=  \int_{0}^1 \int_{\CC}\bigg| \sum_{z_k:|z_k|\geq\eta(R)} a_k r_k(t) f_{(z_k,R)}(z)\bigg|^q d\theta_{(g, \psi ,q )}(z) dt\nonumber\\
&\simeq \int_{0}^1\|V_g F_t\|_{\mathcal{F}_q^\psi}^q dt\lesssim \|(a_k)\|_{\ell^\infty}^q.
\label{series3}
\end{align}
Then, using  \eqref{test0} we get
\begin{align*}
\sum_{z_k:|z_k|\geq\eta(R)} |a_k|^{q}\int_{D(z_k, 3\sigma\tau(z_k))} |g'(z)|^q (1+\psi'(z))^{-q}dm(z) \quad \quad \quad \quad \quad \quad \quad \quad \quad \quad \quad \quad  \quad  \nonumber\\
\lesssim \sum_{z_k:|z_k|\geq\eta(R)}|a_k|^{q}\int_{D(z_k, 3\sigma\tau(z_k))} \frac{|g'(z)|^q |f_{(z_k, R)}(z)|^q}{ (1+\psi'(z))^{q}}e^{-q\psi(z)} dm(z)\nonumber\\
= \int_{\CC} \sum_{z_k:|z_k|\geq\eta(R)}|a_k|^{q}\chi_{D(z_k, 3\sigma\tau(z_k))}(z)| f_{(z_k, R)}(z)|^qd\theta_{(g, \psi, q)}(z)\nonumber\\
\lesssim \max\{1, N_{\max}^{1-q/2}\} \int_{\CC}\Bigg(\sum_{z_k:|z_k|\geq\eta(R)} |a_k|^{2} |f_{(z_k,R)}(z)|^2\Bigg)^{\frac{q}{2}}d\theta_{(g, \psi, q )}(z)\nonumber\\
\lesssim \|(a_k)\|_{\ell^\infty}^q.\ \ \ \quad \quad \quad  \quad
\end{align*}
Setting, in particular,   $a_k=1$  for all $k$ in the above series of estimates results in
\begin{align}
\label{OKK}
\sum_{z_k:|z_k|\geq\eta(R)}\int_{D(z_k, 3\sigma\tau(z_k))} |g'(z)|^q (1+\psi'(z))^{-q}dm(z)< \infty.
\end{align}
Now we take  a positive number $r\geq \eta(R)$ such that whenever  $z_k$ of the covering sequence belongs to
$\{|z|<\eta(R)\}$, then $D(z_k, \sigma\tau(z_k)) $ belongs to $\{|z|<\eta(R)\}$. Thus,
\begin{align}
\int_{\{|w|\geq r\}} \frac{1}{\tau(w)^2}\Bigg(\int_{D(w, \sigma\tau(w))} \frac{ |g'(z)|^q }{(1+\psi'(z))^{q}} dm(z)\Bigg)dm(w)\ \quad \quad \quad \quad \quad \quad \quad \quad \quad \nonumber\\
\leq \sum_{|z_k|\geq\eta(R)}\int_{D(z_k, \sigma\tau(z_k))}\frac{1}{\tau(w)^2}\Bigg(\int_{D(w, \sigma\tau(w))} \frac{ |g'(z)|^q }{(1+\psi'(z))^{q}} dm(z) \Bigg) dm(w)\nonumber\\
\lesssim \sum_{z_k:|z_k|\geq\eta(R)}\int_{D(z_k, 3\sigma\tau(z_k))} \frac{ |g'(z)|^q }{(1+\psi'(z))^{q}} dm(z) < \infty.
\label{got}
\end{align}
On the other hand, applying  \eqref{Paley},  \eqref{test},  and \eqref{compare} we have that
\begin{align*}
\int_{\{|w|< r\}}\frac{1}{\tau(w)^2}\Bigg( \int_{D(w, \sigma\tau(w))} \frac{ |g'(z)|^q }{(1+\psi'(z))^{q}} dm(z)\Bigg) dm(w)\quad \quad \quad \quad  \quad \quad \quad \quad \quad \quad \quad \quad \ \nonumber\\
\simeq \int_{\{|w|< r\}}\frac{1}{\tau(w)^2}\Bigg( \int_{D(w, \sigma\tau(w))} \frac{|g'(z)|^q  |f_{(w,R)}(z)|^q e^{-q\psi(z)}}{(1+\psi'(z))^{q}} dm(z)\Bigg) dm(w)\quad
\nonumber\\
\lesssim \int_{\{|w|< r\}}\frac{\|V_g f_{(w,R)}\|_{\mathcal{F}_q^\psi}^q}{\tau(w)^2} dm(w) \lesssim\int_{\{|w|< r\}}\frac{\| f_{(w,R)}\|_{\mathcal{F}_\infty^\psi}^q}{\tau(w)^2} dm(w)\nonumber\\
 \lesssim \frac{r^2}{\tau(r)^2} <\infty, \quad \quad \quad \quad \quad \quad\quad \quad \quad \quad \quad \quad
 \end{align*} by our admissibility assumption on $\tau$.
This together with \eqref{got}, \eqref{pointwise} as  $|g'|^q$ is subharmonic,  and \cite[Lemma~20]{Olivia} implies
\begin{align*}
\int_{\CC} \frac{ |g'(w)|^q }{(1+\psi'(w))^{q}} dm(w) \lesssim \int_{\CC}\frac{1}{\tau(w)^{2}}\Bigg(\int_{D(w, \sigma\tau(w))}\frac{ |g'(z)|^q }{(1+\psi'(z))^{q}} dm(z)\Bigg)dm(w)< \infty.
\end{align*}
It remains to prove (c) implies (b). Let $f_n$ be a uniformly bounded sequence of functions in $\mathcal{F}_\infty^\psi$ that converges uniformly to zero on   compact subsets of  $\CC$. For each positive $\epsilon$, the necessity of the condition  implies that there exists $N_2$  for which \begin{align*}
\int_{\{|z|>N_2\}}\frac{|g'(z)|^q }{(1+\psi'(z))^q} dm(z) <\epsilon.
\end{align*}
Therefore, it follows from this,  and \eqref{infinite},  and boundedness of $V_g$ that
\begin{align}
\label{fin}
\int_{\{|z| > N_2\}} \frac{ |g'(z)|^q |f_n(z)|^q e^{-q\psi(z)} }{(1+\psi'(z))^{q}} dm(z) \quad \quad \quad \quad \quad \quad \quad \quad \quad \nonumber \\
 \lesssim \|f_n\|_{\mathcal{F}_\infty^\psi}^q\int_{\{|z| >N_2\}} \frac{ |g'(z)|^q   }{(1+\psi'(z))^{q}} dm(z)
\lesssim \epsilon.
\end{align}
We estimate the remaining piece of integral as
\begin{align}
\label{fin1}
\int_{\{|z| \leq N_2\}} \frac{ |g'(z)|^q f_n(z) e^{-q\psi(z)} }{(1+\psi'(z))^{q}} dm(z) \quad \quad \quad \quad \quad \quad \quad \quad \quad \nonumber \\
  \lesssim \sup_{z: |z|\leq N_2}|f_n(z)|^q\int_{\{|z| \leq N_2\}} \frac{ |g'(z)|^q  e^{-q\psi(z)} }{(1+\psi'(z))^{q}} dm(z)\nonumber\\
\lesssim \sup_{ |z|\leq N_2}|f_n(z)|^q \to 0 \ \ \text{as} \ \ n \to \infty.
\end{align}
From \eqref{fin} and \eqref{fin1}, we conclude that $\|V_g f_n\|_{\mathcal{F}_q^\psi }\to 0$.\ \ \ \ \ \ \ \ \ \ \ \ \ $\Box$
\subsection{Proof of Theorem~\ref{thm2}}
\emph{Part (i) and part (ii).}
Since the statement in (c) obviously implies both the statements in  (a) and (b), we plan to show that (a) implies (c) and (b) implies  (c).   Suppose $I_g: \mathcal{F}_p^\psi \to \mathcal{F}_q^\psi $ is bounded.\\
\emph{Case 1:}  We consider first  the case when  $q =\infty.$ Then
 applying \eqref{infinite} and the sequence of  test functions with properties in \eqref{test00}, \eqref{test0}, and  \eqref{test}   we obtain
\begin{align}
\label{good0}
 \|I_g f_{(w, R)}\|_{\mathcal{F}_\infty^\psi} \simeq \sup_{z\in \CC} \frac{|f'_{(w, R)}(z)g(z)|e^{-\psi(z)}}{1+\psi'(z)}\geq \frac{|f'_{(w, R)}(z)g(z)|e^{-\psi(z)}}{1+\psi'(z)}
\end{align} for all $z, w \in \CC$ . Taking $w= z$, as done before,  we have
\begin{align}
\label{good}
\frac{|f'_{(w, R)}(w)g(w)|e^{-\psi(w)}}{1+\psi'(w)} \simeq  |g(w)|.
\end{align}
If $p<\infty$, then combining \eqref{good0}, \eqref{good} and \eqref{test} gives
\begin{align*}
|g(w)| \lesssim \|I_g f_{(w, R)}\|_{\mathcal{F}_\infty^\psi}
 \lesssim \|I_g \|\|f_{(w, R)}\|_{\mathcal{F}_p^\psi}\lesssim \|I_g \| \tau(w)^{2/p}
\end{align*} from which  and \eqref{assumption}, we  see that $|g(w)| \to 0$ as $|w| \to \infty$. Since $g$ is an  analytic function, the above holds only if  it is the zero function.  Similarly, for $p=q= \infty,$  the above techniques shows that  $I_g$ is bounded on $\mathcal{F}_\psi^\infty$ only if $ g$ is a bounded  analytic function, and hence a constant  by  Liouville's  classical theorem.

\emph{Case 2:} $q<\infty$. Then making use of \eqref{test0} and \eqref{test},  we estimate
 \begin{align*}
   \int_{D(w, \sigma \tau(w))}
|g(z)|^q dm(z) \simeq   \int_{D(w, \sigma \tau(w))}
\frac{|f'_{(w, R)}(z)|^q |g(z)|^q}{(1+\psi'(z))^q }e^{-q\psi(z)} dm(z)\quad \quad \quad \nonumber\\
\leq \|V_g f_{(w, R)}\|_{\mathcal{F}_q^\psi}^q \lesssim \|I_g\|^q  \| f_{(w, R)}\|_{\mathcal{F}_p^\psi}^q  \simeq \|I_g\|^q \tau(w)^{\frac{2q}{p}}.
 \end{align*}
 On the other hand,  since  $|g|^q$  is subharmonic,  applying \eqref{pointwise},  and the above  yields
\begin{align}
\label{info1}
|g(w)|^q  \lesssim \frac{1} {\tau(w)^{2}} \int_{D(w,\sigma \tau(w))}|g(z)|^q dm(z)\lesssim \|I_g\|^q \tau(w)^{\frac{2q}{p}-2}
\end{align} from which   we obtain a general necessity condition:
\begin{align}
\label{generalnecessity}
\sup_{w\in \CC} \tau(w)^{2-\frac{2q}{p}} |g(w)|^q \simeq \sup_{w\in \CC} |g(w)|^q (\Delta\psi(w))^{\frac{q-p}{p}} <\infty,
\end{align} for all possible finite exponents $p$ and $ q$. Since by our admissibility assumptions, $\Delta \psi$ increases radially and  in particular for the case  $q\geq p,$  the  condition  in \eqref{generalnecessity} holds  if and only if $ g$ is a bounded  analytic function, and hence a constant.  We  further claim that $g$ is in fact the zero function when $q>p$. If not, then  making use of \eqref{without} and any function $f_p \in \mathcal{F}_p^\psi\setminus \mathcal{F}_q^\psi$  leads to
\begin{align*}
\int_{\CC} \frac{ |f'_{p}(z)|^q}{(1+\psi'(z))^q} e^{-q\psi(z)} dm(z) \simeq \int_{\CC} \frac{|g(z)|^q  |f'_{p}(z)|^q}{(1+\psi'(z))^q} e^{-q\psi(z)} dm(z)\simeq  \|I_g f_p\|_{\mathcal{F}_q^\psi}^q  \nonumber\\
\lesssim \|I_g\|^q \|f_{p}\|_{\mathcal{F}_p^\psi}^q \lesssim \|f_{p}\|_{\mathcal{F}_p^\psi}^q <\infty,
\end{align*}  which shows  that $f_p\in \mathcal{F}_q^\psi$ and gives a contradiction whenever $g$ is a nonzero constant.

It remains to show when $0<q<p<\infty$. For this we  run a  variant of the arguments  used in the proof of part (iii) of Theorem~\ref{thm1}.  We include the details here for the convenience of the reader.  Let  $(r_k(t))_k$  is the Radmecher sequence of function on $[0,1]$  as mentioned before. Then  because of \eqref{finitediscrete} and \eqref{finitediscrete0}  the function
\begin{align*}
F_t= \sum_{z_k:|z_k|\geq\eta(R)} a_k r_k(t) \frac{f_{(z_k,R)}}{\tau(z_k)^{\frac{2}{p}}}
\end{align*} belongs to $\mathcal{F}_\psi^p$ for all $p$ with norm estimate
\begin{align}
\label{series}
\|F_t\|_{\mathcal{F}_\psi^p}^p \lesssim \sum_{k} |a_k|^p.
\end{align}
An application of Khinchine's inequality  again  yields
\begin{align}
\label{Khinchine}
\Bigg(\sum_{z_k:|z_k|\geq\eta(R)} |a_k|^{2} \frac{|f_{(z_k,R)}'(z)|^2}{\tau(z_k)^{\frac{4}{p}}}\Bigg)^{\frac{q}{2}}\lesssim \int_{0}^1\bigg| \sum_{z_k:|z_k|\geq\eta(R)} a_k r_k(t) \frac{f'_{(z_k,R)}(z)}{\tau(z_k)^{\frac{2}{p}}}\bigg|^q dt.
\end{align}
Setting  $
d\theta_{(g, \psi, q)}(z)= |g(z)|^q  e^{-q\psi(z)}(1+\psi'(z))^{-q}dm(z),$ as before  and  making use of \eqref{Khinchine}, and subsequently Fubini's theorem, we have
 \begin{align}
\int_{\CC}\Bigg(\sum_{z_k:|z_k|\geq\eta(R)} |a_k|^{2} \frac{|f_{(z_k,R)}'(z)|^2}{\tau(z_k)^{\frac{4}{p}}}\Bigg)^{\frac{q}{2}}&d\theta_{(g, \psi, q)}(z)\quad \quad \quad \quad \quad  \quad \quad \quad \quad \quad  \quad \quad \quad \quad \quad  \nonumber\\
&\lesssim \int_{\CC} \int_{0}^1\bigg| \sum_{z_k:|z_k|\geq\eta(R)} a_k r_k(t) \frac{f'_{(z_k,R)}(z)}{\tau(z_k)^{\frac{2}{p}}}\bigg|^q dt d\theta_{(g, \psi, q)}(z)\nonumber\\
&=  \int_{0}^1 \int_{\CC}\bigg| \sum_{z_k:|z_k|\geq\eta(R)} a_k r_k(t) \frac{f'_{(z_k,R)}(z)}{\tau(z_k)^{\frac{2}{p}}}\bigg|^q d\theta_{(g, \psi, q)}(z) dt\nonumber\\
&\simeq \int_{0}^1\|I_g F_t\|_{\mathcal{F}_\psi^q}^q dt\lesssim \|(a_k)\|_{\ell^p}^q.
\label{series3}
\end{align}
Now arguing with this, the covering lemma,  and \eqref{test0} leads to the series of estimates
\begin{align*}
\sum_{z_k:|z_k|\geq\eta(R)}\frac{ |a_k|^{q}}{\tau(z_k)^{\frac{2q}{p}}}\int_{D(z_k,3\sigma\tau(z_k))} |g(z)|^q dm(z) \quad \quad \quad \quad \quad \quad \quad \quad \quad \quad \quad \quad  \quad \quad \quad \quad \nonumber\\
\lesssim \sum_{z_k:|z_k|\geq\eta(R)}\frac{ |a_k|^{q}}{\tau(z_k)^{\frac{2q}{p}}}\int_{D(z_k, 3\sigma\tau(z_k))} |g(z)|^q \frac{|f'_{(z_k, R)}(z) |^qe^{-q\psi(z)}}{(1+\psi'(z))^q} dm(z)\nonumber\\
= \int_{\CC} \sum_{z_k:|z_k|\geq\eta(R)}\frac{ |a_k|^{q}}{\tau(z_k)^{\frac{2q}{p}}}\chi_{D(z_k, 3\sigma\tau(z_k))}(z)| f'_{(z_k, R)}(z)|^qd\theta_{(g, \psi, q)}(z)\nonumber\\
\lesssim \max\{1, N_{\max}^{1-q/2}\} \int_{\CC}\Bigg(\sum_{z_k:|z_k|\geq\eta(R)} |a_k|^{2} \frac{|f_{(z_k,R)}'(z)|^2}{\tau(z_k)^{\frac{4}{p}}}\Bigg)^{\frac{q}{2}}d\theta_{(g, \psi, q)}(z)\nonumber\\
\lesssim \|(a_k)\|_{\ell^p}^q. \quad \quad \quad \quad
\end{align*}
Applying duality between the spaces $\ell^{p/q}$ and $\ell^{p/(p-q)}$, we get
\begin{align*}
\sum_{z_k:|z_k|\geq\eta(R)}\Bigg(\frac{1}{\tau(z_k)^2}\int_{D(z_k, 3\sigma\tau(z_k))} |g(z)|^q dm(z)\Bigg)^{p/(p-q)}\tau(z_k)^2 < \infty.
\end{align*}
On the other hand, we can find a positive number $r\geq \eta(R)$ such that whenever a point $z_k$ of the covering sequence $(z_j)$ belongs to
$\{|z|<\eta(R)\}$, then $D(z_k, \sigma\tau(z_k)) $ belongs to $\{|z|<\eta(R)\}$. Thus,
\begin{align*}
\int_{|w|\geq r} \Bigg(\frac{1}{\tau(w)^2}\int_{D(w,3\sigma\tau(w))} |g(z)|^q dm(z)\Bigg)^{p/(p-q)}dm(w)\ \quad \quad \quad \quad \quad \quad \quad \quad \quad \nonumber\\
\leq \sum_{|z_k|\geq\eta(R)}\int_{D(z_k, \sigma\tau(z_k))} \Bigg(\frac{1}{\tau(w)^2}\int_{D(w, 3\sigma\tau(w))} |g(z)|^q dm(z)\Bigg)^{p/(p-q)} dm(w)\nonumber\\
\lesssim \sum_{z_k:|z_k|\geq\eta(R)}\Bigg(\frac{1}{\tau(z_k)^2}\int_{D(z_k, 3\sigma\tau(z_k))} |g(z)|^q dm(z)\Bigg)^{p/(p-q)}\tau(z_k)^2 < \infty.
\end{align*}
It follows from this that
\begin{align}
\label{finiteone}
\int_{|w|< r} \Bigg(\frac{1}{\tau(w)^2}\int_{D(w, 3\delta\tau(w))} |g(z)|^q dm(z)\Bigg)^{p/(p-q)}dm(w)< \infty
\end{align}
By subharmonicity of $|g|^q$ and hence \eqref{pointwise}, we  get the estimate
\begin{align*}
\int_{\CC} |g(w)|^{\frac{qp}{p-q}} dm(w) \lesssim \int_{\CC}\Bigg(\frac{1}{\tau(w)^2}\int_{D(w, 3\sigma\tau(w))} |g(z)|^q dm(z)\Bigg)^{p/(p-q)}dm(w)< \infty
\end{align*}
Since  $g$ is analytic,
the estimate above  holds only if
 $g$ is the zero function as asserted. Interested readers may consult \cite{RA}  to see  why the zero function
 is the only $L^{\frac{qp}{p-q}}$ integrable  entire function on $\CC$.

 Next we prove  (b) implies  (c). We may consider first the case when   $0<q<\infty$. In this case, the problem can be  reformulated in terms of  embedding maps or Carleson measures for the weighted Fock spaces. To this effect,  observe that for any entire function $f$
\begin{align}
\label{multi}
\|M_g f\|_{\mathcal{F}_q^\psi}^q = \int_{\CC} |f(z)|^q |g(z)|^q  e^{-q \psi(z)} dm(z)= \int_{\CC} |f(z)|^q d\mu_{(g,\psi, q)}(z)\nonumber\\
\lesssim \|M_g\|^q \|f\|_{\mathcal{F}_p^\psi}^q,
\end{align} where $d\mu_{(g,\psi, q)}(z)= |g(z)|^q  e^{-q \psi(z)} dm(z).$ This means that the estimate  in \eqref{multi}  holds  if and only if the embedding map $i_d: \mathcal{F}_p^\psi \to L^q(\mu_{(g,\psi)})$ is bounded. By Theorem~1 of \cite{Olivia}, the latter holds if and only if
\begin{align}
\label{oneone}
\sup_{w\in \CC} \frac{1}{\tau(w)^{2q/p}} \int_{D(w, \sigma \tau(w))}|g(z)|^q   dm(z) < \infty
\end{align} for some small  $\sigma >0.$ By Lemma~\ref{lem5}, \eqref{oneone} holds if and only if
\begin{align*}
\sup_{w\in \CC} |g(w)|(\Delta \psi(w))^{\frac{q-p}{p}}<\infty,
\end{align*}
 from which  we easily see that  $g$ is in deed a constant function when $p= q$ and  the zero function otherwise.

On the other hand,  if $0<p< q =\infty,$ then  applying \eqref{test} we have
\begin{align}
\label{mult0}
 \tau(w)^{2/p}\|M_g \|\gtrsim \|M_g f_{(w, R)}\|_{\mathcal{F}_\infty^\psi} =\sup_{z\in \CC}|f_{(w, R)}(z)g(z)|e^{-\psi(z)}\nonumber\\
 \geq |f_{(w, R)}(z)g(z)|e^{-\psi(z)}
\end{align} for all $w$ and $z$ in  $\CC$. In particular,  when $w= z$, \eqref{mult0}, \eqref{info1} and \eqref{test0} give
\begin{align*}
|g(w)| \lesssim |f_{(w, R)}(w)g(w)|e^{-\psi(w)} \simeq  \|M_g \| \tau(w)^{2/p}
\end{align*}from which and \eqref{assumption},  we  again conclude that $g$ is  the  zero function.

\emph{Part (iii}. This can be verified  by arguing as in parts (i) and (ii) of the proofs above and  part (iii) of Theorem~\ref{thm1}.  If $I_g$ is compact on $\mathcal{F}_p^\psi$, then  using the sequence of test functions  $f^*_{(w, R)}$ as defined in \eqref{again}, we easily deduce  that $g$ is  in deed the zero function. On the other hand, if $I_g$ is compact on $\mathcal{F}_p^\psi$, then
taking into account \eqref{multi} and  Theorem~1 of \cite{Olivia} we have that $M_g: \mathcal{F}_p^\psi \to \mathcal{F}_q^\psi $ is compact if and only if
\begin{align*}
\int_{\CC}\Bigg(\frac{1}{\tau(w)^2} \int_{D(w, \sigma \tau(w))} e^{q\psi(z)} d\mu_{(\psi,g)}(z)\Bigg)^{\frac{p}{p-q}}dm(w)\quad \quad \quad \quad \quad \quad \quad \quad \quad\nonumber\\
=\int_{\CC}\Bigg(\frac{1}{\tau(w)^2} \int_{D(w,\sigma \tau(w))} | g(z)|^q dm(z)\Bigg)^{\frac{p}{p-q}}dm(w)<\infty.
\end{align*}
Since $|g|^q$ subharmonic again, by \eqref{pointwise} we have
\begin{align*}
\int_{\CC} |g(w)|^{\frac{p}{p-q}}dm(w) \lesssim \int_{\CC}\Bigg(\frac{1}{\tau(w)^2} \int_{D(w, \sigma \tau(w))} | g(z)|^q dm(z)\Bigg)^{\frac{p}{p-q}}dm(w) <\infty
\end{align*}
Since  $g$ is analytic,
the estimate in above holds only if
 $g$ is the zero function as asserted.
 \ \ \ \ \ \ \ \ \ \ \ \ \ \ $\Box$
\subsection{Proof of Theorem~\ref{thm3}}\label{diff}
 In this subsection, we will verify that the differential operator $D$ is always unbounded whenever it acts between two generalized  Fock spaces.
 Let us assume  that  $D: \mathcal{F}_p^\psi\to \mathcal{F}_q^\psi$ is  bounded and argue in the direction of contradiction. Then,  for simplicity we may split and analyze the situation  in   three different cases.

   \emph{Case 1:} if  $0<p\leq q<\infty$, then  making use of the estimates in \eqref{test0},  \eqref{test},  and \eqref{pointwise}, we have
\begin{align*}
(\tau(w))^{\frac{2q}{p}} \|D\|^q \gtrsim  \|Df_{(w, R)}\|_{\mathcal{F}_q^\psi}^q= \int_{\CC}|f'_{(w, R)}(z)|^q e^{-q\psi(z)} dm(z)\nonumber\\
\geq \int_{D(w, \delta \tau(w))} |f'_{(w, R)}(z)|^q e^{-q\psi(z)} dm(z)\nonumber\\
\gtrsim (\tau(w))^2   |f'_{(w, R)}(w)|^q e^{-q\psi(w)}\simeq (\tau(w))^2  (\psi'(w))^q
\end{align*} for all $w\in \CC$. It follows from this that
\begin{align*}
 \sup_{w\in \CC} (\psi'(w))^q (\Delta \psi(w))^{\frac{q-p}{p}}<\infty
\end{align*} which gives a contradiction since  $\psi$ grows faster than the Gaussian weight function and $\Delta \psi$ radially increases to $\infty.$

\emph{Case 2:} If $0<p<q= \infty$ and $D$ were bounded, then  following a similar argument as above   and  making use of  \eqref{norm2} and \eqref{test}, we would have
\begin{align*}
(\tau(w))^{\frac{2}{p}}\gtrsim\|Df_{(w, R)}\|_{\mathcal{F}_\infty^\psi}= \sup_{z\in \CC}|f'_{(w, R)}(z)| e^{-\psi(z)}\geq |f'_{(w, R)}(z)| e^{-\psi(z)}.
\end{align*}
 This  implies
\begin{align*}
\sup_{w\in \CC} \frac{\psi'(w)}{\tau(w)^{\frac{2}{p}}} <\infty,
\end{align*}
which gives a contradiction as $\tau$ radially decreases to zero and $\psi$ grows faster than the Gaussian weight function again.\\

\emph{Case 3}: if $p= q= \infty$, then arguing towards contradiction as in \emph{case 2}, we have that
\begin{align*}
1\gtrsim\|Df_{(w, R)}\|_{\mathcal{F}_\infty^\psi}= \sup_{z\in \CC}|f'_{(w, R)}(z)| e^{-\psi(z)}\geq |f'_{(w, R)}(z)| e^{-\psi(z)}
\end{align*} for all $z, w\in \CC$. Then as before,  for $w= z$, we get the estimates  $1\gtrsim \psi'(w)$
 which leads to a contradiction when $w \to \infty$.

\emph{ Case 4:} $0<q<p<\infty$.  For this, we modify the arguments from \eqref{series} to \eqref{finiteone} in the proof above.
\begin{align}
\int_{\CC}\Bigg(\sum_{z_k:|z_k|\geq\eta(R)} |a_k|^{2} \frac{|f_{(z_k,R)}'(z)|^2}{\tau(z_k)^{\frac{4}{p}}}\Bigg)^{\frac{q}{2}}e^{-q\psi(z)}dm(z)
\lesssim  \int_{0}^1\|D F_t\|_{\mathcal{F}_\psi^q}^q dt\nonumber\\
\lesssim \|(a_k)\|_{\ell^p}^q.
\end{align}
Now arguing with this, the covering lemma,  and \eqref{test0} leads to the series of estimates
\begin{align*}
\sum_{z_k:|z_k|\geq\eta(R)}\frac{ |a_k|^{q}}{\tau(z_k)^{\frac{2q}{p}}}\int_{D(z_k,3\sigma\tau(z_k))}(1+ |\psi'(z)|)^q dm(z) \quad \quad \quad \quad \quad \quad \quad \quad \quad \quad \quad \quad  \quad \quad \quad \quad \nonumber\\
\simeq  \sum_{z_k:|z_k|\geq\eta(R)}\frac{ |a_k|^{q}}{\tau(z_k)^{\frac{2q}{p}}}\int_{D(z_k, 3\sigma\tau(z_k))} (1+ |\psi'(z)|)^q \frac{|f'_{(z_k, R)}(z) |^qe^{-q\psi(z)}}{(1+\psi'(z))^q} dm(z)\nonumber\\
 \simeq \int_{\CC} \sum_{z_k:|z_k|\geq\eta(R)}\frac{ |a_k|^{q}}{\tau(z_k)^{\frac{2q}{p}}}\chi_{D(z_k, 3\sigma\tau(z_k))}(z)| f'_{(z_k, R)}(z)|^q e^{-q\psi(z)}dm(z)\nonumber\\
\lesssim \max\{1, N_{\max}^{1-q/2}\} \int_{\CC}\Bigg(\sum_{z_k:|z_k|\geq\eta(R)} |a_k|^{2} \frac{|f_{(z_k,R)}'(z)|^2}{\tau(z_k)^{\frac{4}{p}}}\Bigg)^{\frac{q}{2}}e^{-q\psi(z)}dm(z)\nonumber\\
\lesssim \|(a_k)\|_{\ell^p}^q. \quad \quad \quad \quad
\end{align*}
Applying duality between the spaces $\ell^{p/q}$ and $\ell^{p/(p-q)}$, we again get
\begin{align*}
\sum_{z_k:|z_k|\geq\eta(R)}\Bigg(\frac{1}{\tau(z_k)^2}\int_{D(z_k, 3\sigma\tau(z_k))}(1+ |\psi'(z)|)^q dm(z)\Bigg)^{\frac{p}{p-q}}\tau(z_k)^2 < \infty.
\end{align*}
On the other hand, we can find a positive number $r\geq \eta(R)$ such that whenever a point $z_k$ of the covering sequence $(z_j)$ belongs to
$\{|z|<\eta(R)\}$, then $D(z_k, \sigma\tau(z_k)) $ belongs to $\{|z|<\eta(R)\}$. Thus,
\begin{align}
\label{againn}
\int_{|w|\geq r} \Bigg(\frac{1}{\tau(w)^2}\int_{D(w,3\sigma\tau(w))} (1+|\psi'(z)|)^q dm(z)\Bigg)^{\frac{p}{p-q}}dm(w)\ \quad \quad \quad \quad \quad \quad \quad \quad \quad \nonumber\\
\leq \sum_{|z_k|\geq\eta(R)}\int_{D(z_k, \sigma\tau(z_k))} \Bigg(\frac{1}{\tau(w)^2}\int_{D(w, 3\sigma\tau(w))} (1+|\psi'(z)|)^q dm(z)\Bigg)^{\frac{p}{p-q}} dm(w)\nonumber\\
\lesssim \sum_{z_k:|z_k|\geq\eta(R)}\Bigg(\frac{1}{\tau(z_k)^2}\int_{D(z_k, 3\sigma\tau(z_k))} (1+|\psi'(z)|)^q dm(z)\Bigg)^{\frac{p}{p-q}}\tau(z_k)^2 < \infty.
\end{align}
It follows that
\begin{align}
\label{Okagain}
\int_{|w|< r} \Bigg(\frac{1}{\tau(w)^2}\int_{D(w, 3\delta\tau(w))} (1+|\psi'(z)|)^q dm(z)\Bigg)^{\frac{p}{p-q}}dm(w)< \infty.
\end{align}
Taking into account Lemma~20 of \cite{Olivia}, \eqref{againn} and \eqref{Okagain}, we obtain
\begin{align*}
\int_{\CC} \Bigg(\frac{1}{\tau(w)^2}\int_{D(w,3\sigma\tau(w))} (1+|\psi'(z)|)^q dm(z)\Bigg)^{\frac{p}{p-q}}dm(w)\ \ \ \quad \quad \quad \nonumber\\
\simeq \int_{\CC} \frac{(1+|\psi'(z)|)^{\frac{qp}{p-q}}}{\tau(w)^{\frac{2p}{p-q}}} \tau(w)^{\frac{2p}{p-q}}dm(w)= \int_{\CC} (1+|\psi'(z)|)^{\frac{qp}{p-q}}dm(w)<\infty,
\end{align*} which is a contradiction as $p>q$ and $|\psi'(z)| \to \infty $  as $|z| \to \infty$.

\emph{ Case 5:}  $0<q<p= \infty$. For this part, we modify the arguments used in the proof of  (a)
 implies (c) in Theorem~\ref{thm1} and  along with \emph{case 4 } above. Thus we omit the details and leave it to the interested reader.

 It remains to  verify that the differential operator $D$ is always unbounded whenever it acts between two classical  Fock spaces. In this case we can use   the normalized  reproducing kernel $k_{w}(z)= e^{\langle z, w\rangle}-|w|^2/2$  as a test function. If $D$ were bounded, then  for $0<q<\infty$ and $0<p\leq \infty$, we  would get
\begin{align*}
\label{derivative0}
1\gtrsim \|Dk_{w}\|_{\mathcal{F}_q}^q= \int_{\CC}|k'_{w}(z)|^q e^{-\frac{q}{2}|z|^2} dm(z)\geq \int_{D(w,1)}|k'_{w}(z)|^q e^{-\frac{q}{2}|z|^2} dm(z)\nonumber\\
\gtrsim |wk_w(z)|^q e^{-\frac{q}{2}|z|^2}
\end{align*}
 for all $z, w\in \CC$. In particular, setting   $w= z$ here again leads to the estimate
\begin{align*}
1\gtrsim  |w|^q.
\end{align*}
 Letting $|w| \to \infty$ gives a contradiction.  Similarly,  for $q= \infty$, taking $w= z$ again
  which results a contradiction for large $w$.  \ \ $\Box$
 \subsection*{Acknowledgment}
We would like to thank the referee for   careful review of our paper and pointing us relevant literatures, which eventually helped us put our work in context to already known results.

\end{document}